\documentclass[12pt]{amsart}
\newcommand{\ol}{\overline}
\newcommand{\mc}{\mathcal}
\newcommand{\vx}{{\bf x}}
\newcommand{\be}{\begin{equation}}
\newcommand{\ee}{\end{equation}}

\newcommand{\li}{\left<}
\newcommand{\ri}{\right>}

\newtheorem{theorem}{Theorem}[section]
\newtheorem{lemma}[theorem]{Lemma}

\theoremstyle{definition}

\theoremstyle{remark}

\numberwithin{equation}{section}

\begin{document}
\setlength{\baselineskip}{1.2\baselineskip}

\title[The prescribed curvature problem]
{A note on starshaped compact hypersurfaces with prescribed scalar curvature in space forms}

\author{Joel Spruck}
\address{Department of Mathematics, Johns Hopkins University,
 Baltimore, MD, 21218}
\email{js@math.jhu.edu}
\thanks{Research  supported in part by the NSF}
\author{Ling Xiao}
\address{Department of Mathematics, Rutgers University,
 Piscataway, NJ, 08854}
\email{lx70@math.rutgers.edu}

\begin{abstract}
In \cite{GRW}, Guan, Ren and  Wang obtained a $C^2$ a priori estimate
for admissible 2-convex hypersurfaces satisfying the Weingarten curvature equation $\sigma_2(\kappa(X))=f(X, \nu(X)).$
In this note, we give a simpler proof of this result, and extend it to space forms.
\end{abstract}

\maketitle
\section{Introduction}
\label{sec0}
\setcounter{equation}{0}
In \cite{GRW}, Guan,  Ren and  Wang solved the long standing problem of obtaining global $C^2$ estimates for  a closed convex hypersurface $M\subset \mathbb R^{n+1}$ of prescribed kth elementary symmetric function of curvature  in general form:
\be\label{i.1}
\sigma_k(\kappa(X))=f(X, \nu(X)),\,\forall X\in M.
\ee
In the case $k=2$ of scalar curvature, they were able to prove the estimate for strictly starshaped 2-convex hypersurfaces. Their proof relies on  new test curvature functions and elaborate analytic arguments to overcome the difficulties caused by allowing $f$ to depend of $\nu$.

In this note, we give a simpler proof for the scalar curvature case and we extend the result to
space forms $N^{n+1}(K),$ with $K=-1, 0, 1$. Our main result is stated in Theorem \ref{pt1} of section \ref{sec1} and leads to the existence Theorem \ref{thm0}. For related results in the literature see \cite{CNS}, \cite{GLM}, \cite{BLO} and \cite{JL}.

\section{Prescribed scalar curvature}
\label{sec1}
\setcounter{equation}{0}

Let $N^{n+1}(K)$ be a space form of sectional curvature $K=-1,\,0,$ and $+1.$ Let $g^N := ds^2$ denote the Riemannian metric of $N^{n+1}(K).$
In Euclidean space $\mathbb{R}^{n+1},$ fix the origin $O$ and let $\mathbb{S}^n$ denote the unit sphere centered at $O.$
Suppose that $(z, \rho)$ are spherical coordinates in $\mathbb R^{n+1},$ where $z\in\mathbb S^n.$ The standard metric on $S^n$ induced from
$\mathbb{R}^{n+1}$ is denoted by $dz^2.$ Let $a$ be constant, $0<a\leq\infty,$ $I=[0, a),$ and $\phi(\rho)$ a positive function on I. Then the new metric
 \be\label{p.1}
 g^N :=ds^2=d\rho^2+\phi^2(\rho)dz^2.
 \ee
on  $\mathbb{R}^{n+1}$ is a model of  $N^{n+1}$ which is  Euclidean space $\mathbb R^{n+1}$ if $\phi(\rho)=\rho,$ $a=\infty,$
the unit sphere $\mathbb S^{n+1}$ if $\phi(\rho)=\sin(\rho),$ $a=\pi/2$
and  hyperbolic space $\mathbb H^{n+1}$ if $\phi(\rho)=\sinh(\rho),$ $a=\infty.$

We recall some formulas for the induced metric, normal, and second fundamental form on $\mc{M}$ (see \cite{BLO}).
We will denote by $\nabla'$ the covariant derivatives with respect to the standard spherical metric $e_{ij},$ and by $\nabla$
the covariant derivatives with respect to some local orthonormal frame on $\mc{M}.$ Then we have
\be\label{p.4}
g_{ij}=\phi^2e_{ij}+\rho_i\rho_j,\,g^{ij}=\frac{1}{\phi^2}\left(e^{ij}-\frac{\rho^i\rho^j}{\phi^2+|\nabla'\rho|^2}\right),
\ee
\be\label{p.5}
\nu=\frac{(-\nabla'\rho, \phi^2)}{\sqrt{\phi^4+\phi^2|\nabla'\rho|^2}},
\ee
and
\be\label{p.6}
h_{ij}=\frac{\phi}{\sqrt{\phi^2+|\nabla'\rho|^2}}
\left(-\nabla'_{ij}\rho+\frac{2\phi'}{\phi}\rho_i\rho_j+\phi\phi'e_{ij}\right).
\ee

Consider the vector field $V=\phi(\rho)\frac{\partial}{\partial\rho}$ in $N^{n+1}(K),$ and define
$\Phi(\rho)=\int_0^\rho\phi(r)dr.$ Then, $u :=\langle V, \nu \rangle$ is the support function. By a straight forward calculation
we have the following equations (see \cite{GL} lemma 2.2 and lemma 2.6).
\be\label{p.7}
\nabla_{ij}\Phi=\phi'g_{ij}-uh_{ij},
\ee
\be\label{p.8}
\nabla_iu=g^{kl}h_{ik}\nabla_l\Phi,
\ee
and
\be\label{p.9}
\nabla_{ij}u=g^{kl}\nabla_kh_{ij}\nabla_l\Phi+\phi'h_{ij}-ug^{kl}h_{ik}h_{jl}.
\ee

Now let $\Gamma_k$ be the connected component of $\{\lambda\in\mathbb R^n : \sigma_k(\lambda)>0\},$
where \[\sigma_k=\sum_{i_1<i_2<\cdots<i_k}\lambda_{i_1}\cdots\lambda_{i_k}\] is the k-th mean curvature.
$\mathcal{M} :=\{(z, \rho(z)) : z\in\mathbb S^n\}$ is an embedded hypersurface in $N^{n+1}.$ We call
$\rho$ \textit{k-admissible} if the principal curvatures $(\lambda_1(\rho(z)), \dots, \lambda_n(\rho(z)))$ of
$\mathcal{M}$ belong to $\Gamma_k.$ Our problem is to study a smooth positive 2-admissible function
$\rho$ on $\mathbb S^n$ satisfying
\be\label{p.2}
\sigma_2(\lambda(b))=\psi(V,\nu ),
\ee
where $b=\{b_{ij}\}=\{\gamma^{ik}h_{kl}\gamma^{lj}\},$ $\{h_{ij}\}$ is the second fundamental form of $\mathcal{M},$ and
$\gamma^{ij}$ is $\sqrt{g^{-1}}.$
Equivalently, we  study the solution of the following equation
\be\label{p.3}
F(b)= \left( \begin{array}{c}n\\2 \end{array} \right)^{(-1/2)}\sigma_2(\lambda(b))^{1/2}
=f(\lambda(b_{ij}))=\overline{\psi}(V,\,\nu).
\ee

Now we are ready to state and prove our main result.
\begin{theorem}\label{pt1}
Suppose $\mc{M}=\{(z,\,\rho(z))\mid z\in\mathbb S^n\}\subset N^{n+1}$ is a closed 2-convex hypersurface which is strictly starshaped  with respect to the origin and satisfies equation \eqref{p.3} for some positive function $\ol\psi(V,\,\nu)\in C^2(\Gamma),$
where $\Gamma$ is an open neighborhood of the unit normal bundle of $\mc{M}$ in $N^{n+1}\times\mathbb S^n.$
Suppose also we have uniform control $0<R_1\leq \rho(z)\leq R_2<a,\, |\rho|_{C^1}\leq R_3$. Then there is a constant $C$ depending only on $n,R_1,R_2,R_3$  and $|\bar{\psi}|_{C^2},$
such that
\be\label{p.10}
\max_{z\in\mathbb{S}^n}|\kappa_i(z)|\leq C.
\ee
\end{theorem}
\begin{proof}
Since $\sigma_1(\kappa)>0$ on $\mc{M}$, it suffices to estimate from above, the largest principal curvature  of $\mc{M}$. Consider
 \[M_0=\max_{\vx \in \mc{M}}e^{\beta\Phi}\frac{\kappa_{\max}}{u-a}~,\]
  where $u \geq 2a$ and $\beta$ is a large constant to be chosen (we will always assume $\beta \phi'+K>0$).
  Then $M_0$ is achieved  at $\vx_0=(z_0,\,\rho(z_0))$
and we may choose a local orthonormal frame $e_1, \dots, e_n$ around $\vx_0$ such that  $h_{ij} (\vx_0) = \kappa_i \delta_{ij}$,
where $\kappa_1, \ldots, \kappa_n$ are the principal curvatures of $\Sigma$
at $\vx_0$. We may assume
$\kappa_1 = \kappa_{\max} (\vx_0)$.
Thus at $\vx_0,\, \log{h_{11}}-\log{(u-a)}+ \beta\Phi$
has a local maximum.  Therefore,
\be\label{p.11}
0=\frac{\nabla_ih_{11}}{h_{11}}-\frac{\nabla_iu}{u-a}+\beta\Phi_i,
\ee
and
\be\label{p.12}
0\geq\frac{\nabla_{ii}h_{11}}{h_{11}}-\left(\frac{\nabla_ih_{11}}{h_{11}}\right)^2
-\frac{\nabla_{ii}u}{u-a}+\left(\frac{\nabla_iu}{u-a}\right)^2+\beta\Phi_{ii}.
\ee
By the Gauss and Codazzi equations, we have
$\nabla_k h_{ij}=\nabla_j h_{ik}$ and
\be\label{codazzi}
\nabla_{11}h_{ii}=\nabla_{ii}h_{11}+h_{11}h_{ii}^2-h_{11}^2h_{ii}+K(h_{11}\delta_{1i}\delta_{1i}-h_{11}\delta_{ii}+h_{ii}-h_{i1}\delta_{i1}).
\ee
Therefore,
\be\label{p.13}
\begin{aligned}
F^{ii}\nabla_{11}h_{ii}&=F^{ii}\nabla_{ii}h_{11}+\kappa_1\sum_if_i\kappa_i^2
-\kappa_1^2\sum f_i\kappa_i+K\left(-\kappa_1\sum_if_i+\sum_if_i\kappa_i\right)\\
&=\sum_if_i\nabla_{ii}h_{11}+\kappa_1\sum_if_i\kappa_i^2-\ol{\psi}\kappa_1^2+K\left(-\kappa_1\sum_if_i+\ol{\psi}\right)\\
\end{aligned}
\ee

Covariantly differentiating equation \eqref{p.3}  twice yields
\be\label{p.1'}
F^{ii}h_{iik}=\bar{\psi}_V(\nabla_{e_k}V)+h_{ks}\bar{\psi}_{\nu}(e_s)\\
\ee
so that
 \be \label{p.1''}
 |\sum_i f_i h_{iis}\Phi_s| \leq C(1+\kappa_1)
\ee
and
\be\label{p.2'}
\begin{aligned}
F^{ii}h_{ii11}+F^{ij,kl}h_{ij1}h_{kl1}&=\nabla_{11}(\ol{\psi})
\geq-C(1+\kappa_1^2)+h_{11s}\bar{\psi}_{\nu}(e_s)\\
&\geq -C(1+\kappa_1^2 +\beta \kappa_1) \hspace{.1in}\text{(using \eqref{p.11})}.
\end{aligned}
\ee

Combining \eqref{p.2'} and \eqref{p.13} and using
\eqref{p.7}, \eqref{p.8},\eqref{p.9},\eqref{p.11},\eqref{p.12}, \eqref{p.1'},\eqref{p.1''} gives
\be\label{p.14}
\nonumber
\begin{aligned}
0&\geq\frac1{\kappa_1}\left\{-C(1+\kappa_1^2+\beta \kappa_1)-F^{ij,kl}\nabla_1h_{ij}\nabla_1h_{kl}-\kappa_1\sum f_i\kappa_i^2
+\kappa_1^2 \ol{\psi}-K(-\kappa_1\sum f_i+\ol{\psi})\right\}\\
&-\frac{1}{\kappa_1^2}\sum f_i|\nabla_ih_{11}|^2-\frac{1}{u-a}\sum f_i\{h_{iis}\Phi_s-u\kappa_i^2+\phi'\kappa_i\}
+\sum f_i\frac{|\nabla_iu|^2}{(u-a)^2}-u\beta\ol{\psi}+\beta\phi'\sum f_i\\  %
&\geq -C(\kappa_1+\beta )-\frac1{\kappa_1}F^{ij,kl}\nabla_1h_{ij}\nabla_1h_{kl}+\frac{a}{u-a}\sum f_i\kappa_i^2+(\beta \phi'+K)\sum f_i\\ %
&-\frac{1}{\kappa_1^2}\sum f_i|\nabla_ih_{11}|^2
+\sum f_i\frac{|\nabla_iu|^2}{(u-a)^2}\\
\end{aligned}
\ee
In other words,
\be\label{p.16}
\begin{aligned}
0&\geq-C(\kappa_1+\beta)-\frac{1}{\kappa_1}F^{ij,kl}\nabla_1h_{ij}\nabla_1h_{kl}+\frac{a}{u-a}\sum f_i\kappa_i^2\\
&+(\beta\phi'+K)\sum f_i-\frac{1}{\kappa_1^2}\sum f_i|\nabla_ih_{11}|^2+\sum f_i\frac{|\nabla_iu|^2}{(u-a)^2}.\\
\end{aligned}
\ee
By \eqref{p.11} we have for any $\epsilon>0$,
\be\label{p.17}
\frac{1}{\kappa_1^2}\sum f_i|\nabla_ih_{11}|^2\leq
(1+\epsilon^{-1})\beta^2\sum f_i|\nabla_i\Phi|^2+\frac{(1+\epsilon)}{(u-a)^2}\sum f_i|\nabla_iu|^2.
\ee
Using this in \eqref{p.16} we obtain
\be\label{p.18}
\begin{aligned}
0&\geq-C(\kappa_1+\beta)-\frac{1}{\kappa_1}F^{ij,kl}\nabla_1h_{ij}\nabla_1h_{kl}+(\frac{a}{u-a}-C\epsilon)\sum f_i\kappa_i^2\\
&+[\beta\phi'+K-C\beta^2(1+\epsilon^{-1})]T~,
\end{aligned}
\ee
where $T=\sum f_i~.$
Now we divide the remainder of the proof into two cases.

Case A. Assume $\kappa_n\leq-\frac{\kappa_1}n$.
In this case, equation \eqref{p.18} implies (here $\epsilon$ is small  controlled multiple of $a$
 and we use $f_n\geq f_i$ which holds by concavity of $f$)
\be\label{p.19}
0\geq-C(\kappa_1+\beta)+\frac{a}{C}\sum f_i\kappa_i^2-C\beta^2 T
\geq-C(\kappa_1+\beta)+(\frac1C \kappa_1^2-C\beta^2)T
\ee
Since $T\geq 1$ by the concavity of $f$,
equation \eqref{p.19} implies $\kappa_1\leq C\beta$ at $\vx_0$.

Case B. Assume $\kappa_n>-\frac{\kappa_1}n$. Let us partition $\{1, \cdots, n\}$ into 2 parts,
\[I=\{j: f_j\leq n^2f_1\}\, \mbox{and}\, J=\{j: f_j>n^2 f_1\}.\]
For $i\in I,$ we have (by \eqref{p.11}) for any $\epsilon>0 $
\be\label{p.20}
\frac{1}{\kappa_1^2}f_i|\nabla_ih_{11}|^2\leq(1+\epsilon)\sum f_i\frac{|\nabla_iu|^2}{(u-a)^2}
+C(1+\epsilon^{-1})\beta^2f_1.
\ee
Inserting this  into equation \eqref{p.16} gives (for $\epsilon$ a small controlled multiple of $a^2$)
\be\label{p.21}
\begin{aligned}
0&\geq-C(\kappa_1+\beta)-\frac{1}{\kappa_1}F^{ij,kl}\nabla_1h_{ij}\nabla_1h_{kl}+\frac{a}{C}\sum f_i\kappa_i^2+(\beta\phi'+K)\sum f_i\\
&-\frac{1}{\kappa_1^2}\sum_{i\in J}f_i|\nabla_ih_{11}|^2-C\beta^2f_1.
\end{aligned}
\ee
Now we  use an inequality due
to Andrews~\cite{Andrews94} and Gerhardt~\cite{Gerhardt96}:
\be\label{p.22}
\begin{aligned}
 - &\frac1{\kappa_1}F^{ij,kl} \nabla_1 h_{ij} \nabla_1h_{kl}
 \geq \frac1{\kappa_1} \sum_{i \neq j} \frac{f_i - f_j}{\kappa_j - \kappa_i} |\nabla_1 h_{ij}|^2\\
 &\geq \frac2{\kappa_1} \sum_{j\geq 2} \frac{f_j - f_1}{\kappa_1 - \kappa_j} |\nabla_j h_{11}|^2\\
&\geq\frac{2}{\kappa_1^2}\sum_{j\in J}f_j|\nabla_jh_{11}|^2~.
\end{aligned}
\ee
We now
insert \eqref{p.22} into \eqref{p.21} to obtain
\be\label{p.23}
0\geq-C(\kappa_1+\beta)+\frac{a}{C}\sum f_i\kappa_i^2+(\beta\phi'+K)\sum f_i-C\beta^2 f_1.
\ee
Since  $\kappa_n>-\frac1n \kappa_1$ we have that
\[\sum f_i=\frac{(n-1)\sigma_1} {2{n \choose 2}\ol{\psi}}>\frac{\kappa_1-\frac{n-1}n\kappa_1}{n\ol{\psi}}=
\frac{\kappa_1}{n^2 \ol{\psi}}\]

We also note that on $\mc{M},\, \phi'$ is bounded below by a positive controlled constant so we may assume $\beta \phi'+K$ is large. Therefore from \eqref{p.23}
we obtain
\be\label{p.24}
0\geq (\frac{\beta \phi'+K}{n^2 \ol{\psi}}-C)\kappa_1-C\beta+
(\frac{a}{C_2}\kappa_1^2-C\beta^2 )f_1.
\ee
We now fix $\beta$ large enough that $\frac{\beta \phi'+K}{n^2 \ol{\psi}}>2C$ which implies a uniform upper bound for $\kappa_1$ at $\vx_0$. By the definition of $M_0$
we then obtain a uniform upper bound for $\kappa_{max}$ on $\mc{M}$ which implies a uniform upper and lower bound for the principle curvatures.

\end{proof}

\medskip

\medskip

\section{Lower order estimates}
\label{sec2}
\setcounter{equation}{0}

In this section, we obtain $C^0$ and $C^1$ estimates for the more general equation:
\be\label{k.1}
\sigma_{k}(\kappa)=\psi(V,\,\nu),
\ee
where $k=1, \cdots, n.$
\subsection{$C^0$ estimates}

The $C^0$-estimates were proved in \cite{BLO} but  for the reader's convenience we include the simple proof.
\begin{lemma}\label{lem1}
Let $1\leq k\leq n$ and let $\psi\in C^2(N^{n+1}\times\mathbb{S}^n)$ be a positive function.
Suppose there exist two numbers $R_1$ and $R_2,$ $0<R_1<R_2<a,$ such that
\be\label{l.1}
\psi\left(V, \frac{V}{|V|}\right)\geq\sigma_k(1,\cdots, 1)q^k(\rho), \rho=R_1,
\ee
\be\label{l.2}
\psi\left(V, \frac{V}{|V|}\right)\leq\sigma_k(1,\cdots, 1)q^k(\rho), \rho=R_2,
\ee
where $q(\rho)=\frac{1}{\phi}\frac{d\phi}{d\rho}.$ Let $\rho\in C^2(\mathbb{S}^n)$ be a solution of equation \eqref{k.1}. Then
\[R_1\leq\rho\leq R_2.\]
\end{lemma}
\begin{proof} Suppose that $\max_{z\in\mathbb S^n}\rho(z)=\rho(z_0)>R_2$. Then at $z_0$,
\[g^{ij}=\phi^{-2}e^{ij},\,h_{ij}=-\nabla'_{ij} \rho+\phi\phi' e_{ij}\geq \phi \phi' e_{ij}, b_{ij}\geq q(\rho) \delta_{ij}~.\]
Hence $\psi(V,\,\nu)(z_0)=\sigma_k(b_{ij})(z_0)>q^k(R_2)\sigma_k(1,\ldots,1)$, contradicting \eqref{l.2}.
The proof of \eqref{l.1} is similar.
\end{proof}

\subsection{$C^1$ estimates}
In this section, we follow the idea of \cite{CNS} and \cite{GLM}
to derive $C^1$ estimates for the height function $\rho.$ In other words, we are looking for a lower bound
for the support function $u.$  First, we need the following technical assumption:
for any fixed unit vector $\nu,$
\be\label{l.3}
\frac{\partial}{\partial\rho}(\phi(\rho)^k\psi(V, \nu))\leq 0,\,\,\mbox{where $|V|=\phi(\rho).$}
\ee
\begin{lemma}\label{lem2}
Let $M$ be a radial graph in $N^{n+1}$ satisfying \eqref{k.1},\eqref{l.3} and let $\rho$ be the height function of $M.$
If $\rho$ has positive upper and lower bounds, then there is a constant C depending on the minimum and maximum values
of $\rho,$ such that
\[|\nabla\rho|\leq C.\]
\end{lemma}
\begin{proof}
Consider $h=-\log u+\gamma(\Phi(\rho))$ and suppose $h$ achieves it's maximum at $z_0.$ We will show that for a suitable choice of $\gamma(t),\, u(z_0)=|V(z_0)|$, that is $V(z_0)=|V(z_0)|\nu(z_0)$, which implies a uniform lower bound for $u$ on $M$.
If not, we can choose a local orthonormal
frame $\{e_1,\cdots, e_n\}$ on M such that $\left<V, e_1\right>\neq 0,$ and $\left<V, e_i\right>=0,\, i\geq 2.$
Then at $z_0$ we have,
\be\label{l.4}
h_i=\frac{-u_i}{u}+\gamma'\nabla_i\Phi=0,
\ee
\be\label{l.5}
\begin{aligned}
&0\geq h_{ii}=\frac{-u_{ii}}{u}+\left(\frac{u_i}{u}\right)^2+\gamma'\nabla_{ii}\Phi
+\gamma''(\nabla_i\Phi)^2\\
&=\frac{-1}{u}\left(h_{ii1}\nabla_1\Phi+\phi'h_{ii}-uh_{ii}^2\right)
+[(\gamma')^2+\gamma''](\nabla_i\Phi)^2+\gamma'(\phi'g_{ii}-h_{ii}u).\\
\end{aligned}
\ee
Equation \eqref{l.4} gives
\be\label{l.8}
h_{11}=u\gamma',\, h_{i1}=0, \,i\geq 2
\ee
so we may rotate $\{e_2,\cdots, e_n\}$ so that ${h_{ij}}(z_0,\rho(z_0)$ is diagonal.
Hence,
\be\label{l.6}
\begin{aligned}
0&\geq\frac{-1}{u}\left(\sigma_k^{ii}h_{ii1}\nabla_1\Phi+\phi'k\psi-u\sigma_k^{ii}h_{ii}^2\right)\\
&+[(\gamma')^2+\gamma''](\nabla_1\Phi)^2\sigma^{11}+\gamma'(\phi'\sum \sigma_k^{ii}-k\psi u)\\
\end{aligned}
\ee
Differentiating equation \eqref{k.1} with respect to $e_1$ we obtain
\be\label{l.7}
\sigma_k^{ii}h_{ii1}=d_V\psi(\nabla_{e_1}V)+h_{11}d_{\nu}\psi(e_1).
\ee

Substituting equation \eqref{l.7} and \eqref{l.8} into \eqref{l.6} yields
\be\label{l.9}
\begin{aligned}
0&\geq \frac{-1}{u}[\li V, e_1\ri d_V\psi(\nabla_{e_1}V)+u\gamma'\li V, e_1\ri d_{\nu}\psi(e_1)+k\phi'\psi]\\
&+\sigma_k^{ii}h_{ii}^2+[(\gamma')^2+\gamma'']\li V, e_1\ri^2\sigma_k^{11}+\gamma'\phi'\sigma_k^{ii}-ku\gamma'\psi\\
&=\frac{-1}{u}[\li V, e_1\ri d_V\psi(\nabla_{e_1}V)+k\phi'\psi]+\sigma_k^{ii}h_{ii}^2\\
&+[(\gamma')^2+\gamma'']\li V, e_1\ri^2\sigma_k^{11}+\gamma'\phi' \sum \sigma_k^{ii}-u\gamma'\psi-\gamma'\li V, e_1\ri d_{\nu}\psi(e_1).
\end{aligned}
\ee
Our assumption \eqref{l.3} is equivalent to
\be\label{l.10}
k\phi^{k-1}\phi'\psi+\phi^{k}\frac{\partial}{\partial\rho}\psi(V, \nu)\leq 0,
\ee
or
\be\label{l.11}
k\phi'\psi+d_V\psi(V, \nu)\leq 0.
\ee
Since at $z_0,$ $V=\li V, e_1\ri e_1+\li V, \nu\ri\nu$
\be\label{l.12}
d_V\psi(V, \nu)=\li V, e_1\ri d_V\psi(\nabla_{e_1}V)+\li V, \nu\ri d_V\psi(\nabla_{\nu}V).
\ee
Therefore,
\be\label{l.13}
\begin{aligned}
0&\geq\sigma_k^{ii}h_{ii}^2+[(\gamma')^2+\gamma'']\li V, e_1\ri^2\sigma_k^{11}+\gamma'\phi' \sum \sigma_k^{ii}\\
&-u\gamma'\psi-\gamma'\li V, e_1\ri d_{\nu}\psi(e_1)+d_V\psi(\nabla_{\nu}V)
\end{aligned}
\ee
Now let $\gamma(t)=\frac{\alpha}{t},$ where $\alpha>0$ is sufficiently large.
Since $h_{11}\leq 0$ at $z_0,$ and $\sum \sigma_k^{ii}=(n-k+1)\sigma_{k-1}$,
we have that
\be\label{l.15}
\sigma_k^{11}=\sigma_{k-1}(\kappa|\kappa_1)\geq \sigma_{k-1}\geq\sigma_k^{\frac{k-1}{k}}=\psi^{\frac{k-1}{k}}.
\ee

Therefore
\be\label{l.14}
[(\gamma')^2+\gamma'']\li V, e_1\ri^2\sigma_k^{11}+\sigma_{k}^{ii}h_{ii}^2+\gamma'\phi' \sum\sigma_k^{ii}
\geq C\alpha^2\sigma_k^{11},
\ee
for some $C$ depending on $|\rho|_{C^0}.$

We conclude that
\be\label{l.16}
0\geq C\alpha^2\psi^{\frac{k-1}{k}}-\alpha|V||d_{\nu}\psi(e_1)|-|d_V\psi(\nabla_{\nu}V)|,
\ee
which leads to a contradiction when $\alpha$ is large. Therefore at $z_0$ we have $u=|V|,$ which completes the proof.
\end{proof}

By a standard continuity argument (see \cite{CNS}), we can prove the following theorem.
\begin{theorem}\label{thm0}
Suppose $\psi\in C^2(\bar{B}_{r_2}\setminus B_{r_1}\times\mathbb{S}^n)$ satisfies
conditions \eqref{l.1}, \eqref{l.2}, and \eqref{l.3}. Then there exists a unique $C^{3,\alpha}$
starshaped solution $\mathcal{M}$ satisfying equation \eqref{p.2}.
\end{theorem}

\end{document}